\documentclass[10pt]{article}
\usepackage{amssymb}
\usepackage{amsthm}
\usepackage{graphicx}
\newtheorem{theorem}{Theorem}
\newtheorem{lemma}{Lemma}
\newtheorem{corollary}{Corollary}
\newtheorem{remark}{Remark}

\newtheorem{conjecture}{Conjecture}

\def\Frac#1#2{\frac{\displaystyle{#1}}{\displaystyle{#2}}}
\def\cali{{\cal{I}}}
\begin{document}
 \title{Monotonicity properties for ratios and products of modified Bessel functions and sharp trigonometric bounds}

\author{J. Segura\\
        Departamento de Matem\'aticas, Estad\'{\i}stica y 
        Computaci\'on,\\
        Universidad de Cantabria, 39005 Santander, Spain.\\
        javier.segura@unican.es
}

\date{ }

\maketitle
\begin{abstract}
Let $I_{\nu}(x)$ and $K_{\nu}(x)$ be the first and second kind modified Bessel functions. 
It is shown that the nullclines of the Riccati equation satisfied by 
$x^{\alpha} \Phi_{i,\nu}(x)$, $i=1,2$, with $\Phi_{1,\nu}=I_{\nu-1}(x)/I_{\nu}(x)$ and 
$\Phi_{2,\nu}(x)=-K_{\nu-1}(x)/K_{\nu}(x)$, 
are bounds for $x^{\alpha} \Phi_{i,\nu}(x)$, which are solutions with unique monotonicity properties;
these bounds hold at least for $\pm \alpha\notin (0,1)$ and $\nu\ge 1/2$.  
Properties for the product 
$P_{\nu}(x)=I_{\nu}(x)K_{\nu}(x)$ can be obtained as a consequence; for instance, 
it is shown that $P_{\nu}(x)$ is decreasing
if $\nu\ge -1$ (extending the known range of this result) and that $xP_{\nu}(x)$ is
increasing for $\nu\ge 1/2$. We also show that  the double 
ratios $W_{i,\nu}(x)=\Phi_{i,\nu+1}(x)/\Phi_{i,\nu}(x)$ are monotonic and that these monotonicity properties
 are exclusive of the first and second kind modified Bessel functions. Sharp trigonometric bounds can be
 extracted from the
 monotonicity of the double ratios. The trigonometric bounds for the ratios and the
product
are very accurate as $x\rightarrow 0^+$, $x\rightarrow +\infty$
and $\nu\rightarrow +\infty$ in the sense that the first two terms in the power series 
expansions in these limits  are exact.
\end{abstract}



\section{Introduction}

Modified Bessel function, and in particular their ratios, are important special functions appearing in
countless applications. Bounds for these ratios are needed in a huge number of different 
scientific and engineering fields,
like finite elasticity \cite{Simpson:1984:SMR}, 
telecommunications \cite{Azari:2018:URU}, statistics \cite{Igarashi:2014:RFO}, heat 
transfer \cite{Calvo:2018:ASB},  information theory 
\cite{Dytson:2020:TCA} and many others. Not surprisingly, this
is an active topic of study; see for instance 
\cite{Baricz:2009:OAP,Laforgia:2010:SIF,
Segura:2011:BFR,Hornik:2013:ABF,Baricz:2015:BFT,Ruiz:2016:ANT,
Yang:2017:TMA,Yang:2018:MAC}. 

In this paper we obtain new monotonicity properties and bounds for ratios and products of modified Bessel functions,
 some of them displaying a remarkable accuracy in all three directions as $x\rightarrow 0^+$, $x\rightarrow +\infty$
and $\nu\rightarrow +\infty$; we also extend previous results, in particular for the product of first and second kind modified Bessel functions.

We analyze the monotonicity of the functions 
$x^{\alpha} \Phi_{i,\nu}(x)$, with $\Phi_{1,\nu}(x)=I_{\nu-1}(x)/I_{\nu}(x)$ and 
$\Phi_{2,\nu}(x)=-K_{\nu-1}(x)/K_{\nu}(x)$, by considering the Riccati equation satisfied by these functions.
It is shown that the nullclines of the Riccati equation are bounds for $x^\alpha\Phi_{i,\nu}(x)$, at least when 
$\pm\alpha\notin (0,1)$ and $\nu\ge 1/2$.
We show that these monotonicity properties are unique for the first
 and second kind Bessel functions and no other solution of the Riccati equation
 is both regular and monotonic when $\nu\ge 1$. The bounds for the ratios of Bessel functions that
can be obtained as a consequence of this analysis
 are described and then applied to the study of
the monotonicity and bounds for the product $I_{\nu}(x)K_{\nu}(x)$. We prove that  
$I_{\nu}(x)K_{\nu}(x)$ is decreasing if $\nu\ge -1$ (enlarging the range of validity considered so far) while
$xI_{\nu}(x)K_{\nu}(x)$ is increasing for $\nu\ge 1/2$. Upper and lower bounds for the product are also made available.

In a similar way, the monotonicity properties of the double ratios
$W_{i,\nu}(x)=\Phi_{i,\nu+1}(x)/\Phi_{i,\nu}(x)$ are established and proved to be 
unique for the first and second kind modified Bessel functions. New sharp trigonometric bounds for both 
the first and second kind modified Bessel functions ratios are obtained from this analysis. 
These bounds, both for the ratios and the products, are shown to be very accurate in the three limits 
$x\rightarrow 0^+$, $x\rightarrow +\infty$ and $\nu\rightarrow +\infty$, in the sense that at least the
 first two terms of the power series expansions of the ratios and products, in any of these limits,
 is given exactly by our new bounds.

The main tool for proving these results is the analysis of the qualitative properties of the first order
differential equations satisfied by the ratios and double ratios of Bessel functions. For the case of the single
ratios, this analysis is similar to that of \cite{Segura:2011:BFR,Ruiz:2016:ANT}; we summarize some of these 
results in section \ref{riccatis}, we discuss how the monotonicity properties are unique for the first and second kind functions (and therefore the bounds are sharp only for such functions) and we prove the monotonicity 
properties for the product $P_{\nu}(x)$ and the corresponding bounds. In section \ref{beyond}
 we study the 
monotonicity of the double ratio by considering the first order differential equation satisfied by this ratio.
In this analysis, the nullclines of the differential equation satisfied by the double ratio, which are solutions of an algebraic cubic equation, will be shown
to provide very sharp bounds for the simple ratios 
$\Phi_{i,\nu}(x)$ (similarly as happened in \cite{Segura:2020:UVS} for Parabolic Cylinder functions) 
  and then, as a consequence, for the double ratio $W_{i,\nu}(x)$ and the
product $P_{\nu}(x)$.

\section{Bounds from the Riccati equation}

\label{riccatis}

The starting point in the analysis is the difference-differential system 
\cite[10.29.2]{Olver:2010:BF}
\begin{equation}
\label{DDE}
\begin{array}{l}
\cali'_{\nu}(x)=\cali_{\nu-1}(x)-\Frac{\nu}{x}\cali_{\nu}(x)\\
\\
\cali'_{\nu-1}(x)=\cali_{\nu}(x)+\Frac{\nu -1}{x}\cali_{\nu-1}(x),
\end{array}
\end{equation}
which is satisfied by $I_{\nu}(x)$ and $e^{i \pi \nu} K_{\nu}(x)$ 
\footnote{The complex notation is
not substantial and we could have also defined a second solution for real $\nu$ as 
$(-1)^{\lfloor \nu \rfloor}K_{\nu}(x)$}, which as a consequence also 
satisfy
\begin{equation}
\label{TTRR}
\cali_{\nu+1}(x)+\Frac{2\nu}{x}\cali_\nu (x)-\cali_{\nu-1}(x)=0.
\end{equation}
$\cali_{\nu}(x)=I_{-\nu}(x)$ is also a solution of (\ref{DDE}), but it is not independent of $I_\nu (x)$ 
for integer $\nu$.

For proving the results in this paper, the only information which will be needed as input is the difference-differential
 system (\ref{DDE}) together with information on the sign of the function ratios and first derivatives as 
$x\rightarrow 0^+$ and $x\rightarrow +\infty$; this information 
will single out two of the solutions of the system (\ref{DDE}),
 specifically the regular solution at $x=0$ ($\cali_\nu (x)=I_\nu (x)$) and the recessive solution as 
$ x\rightarrow +\infty$ ($\cali_\nu (x)=e^{i\pi\nu}K_\nu (x)$).

We first briefly review how bounds for the ratios of first and second kind modified 
Bessel functions can be obtained by analyzing the nullclines
of the Riccati equations satisfied by these ratios (Theorem \ref{moik}), as done in \cite{Segura:2011:BFR,Ruiz:2016:ANT}. 
The analysis of the bounds for the first and second kind Bessel
functions will be done simultaneously using a same Riccati equation, which differs slightly 
from the approach in \cite{Segura:2011:BFR,Ruiz:2016:ANT}.
Later in this section we 
consider the more general case of the general solution of the system (\ref{DDE}) and we study the monotonicity
 properties and bounds for the product $P_{\nu}(x)=I_{\nu}(x)K_{\nu}(x)$.

Starting from the DDE (\ref{DDE}) we can obtain the Riccati equation for 
\begin{equation}
\Phi_{\nu}(x)=\Frac{\cali_{\nu-1}(x)}{\cali_{\nu}(x)},
\end{equation}
giving
\begin{equation}
\label{Riccati}
\Phi'_{\nu}(x)=1+\Frac{2\nu-1}{x}\Phi_{\nu}(x)-\Phi_{\nu}(x)^2.
\end{equation}

Using $\{I_{\nu}(x),(-1)^{\lfloor \nu \rfloor} K_\nu (x)\}$ 
as a pair of independent solutions of the DDE (\ref{DDE}), we can write the solutions $\Phi_\nu (x)$ 
as 
\begin{equation}
\label{gene}
\Phi_\nu (x) \equiv 
\Phi_{t, \nu} (x) =\Frac{\cos\left(\frac{\pi}{2}t\right) 
I_{\nu -1}(x)- \sin\left(\frac{\pi}{2}t\right) K_{\nu -1}(x)}
{\cos\left(\frac{\pi}{2}t\right) I_{\nu}(x)+\sin\left(\frac{\pi}{2}t\right) K_{\nu}(x)},
\, t\in (-1,1],
\end{equation}

$\Phi_{t,\nu}(x)$ is the general solution of (\ref{Riccati}). That it is a solution is obvious by construction and
 that any solution can be written in this form is also clear: for any $\nu$ and 
for each $(x,y)$ there is one and only one 
solution of (\ref{Riccati}) such that $\Phi_{\nu}(x)=y$ and there exists a unique value of $t\in (-1,1]$ such
 that $\Phi_{t,\nu}(x)=y$, and $\Phi_{t,\nu}(x)$ is precisely this unique solution.

As in \cite{Ruiz:2016:ANT}, we consider a more general Riccati equation by taking 
\begin{equation}
\gamma_{t,a,\nu} (x)=x^{-a}\Phi_{t,\nu}(x). 
\end{equation}
We have
\begin{equation}
\label{rca2}
\begin{array}{ll}
\gamma'_{t,a,\nu}(x)&=-x^a \gamma_{t,a,\nu}(x)^2+\Frac{2\nu-1-a}{x}\gamma_{t,a,\nu} (x)+x^{-a}\\
&\\
&=-x^{a}(\gamma_{t,a,\nu}(x)-\hat{\gamma}_{a, \nu}^+(x))(\gamma_{t,a,\nu}(x)-\hat{\gamma}_{a,\nu}^-(x))
\end{array}
\end{equation}
where
\begin{equation}
\hat{\gamma}_{a,\nu}^+(x)=x^{-a}\lambda_{a,\nu}^+ (x),\,
\hat{\gamma}_{a,\nu}^-(x)=-x^{-a}/\lambda_{a,\nu}^+ (x)
\end{equation}
with
\begin{equation}
\lambda_{a,\nu}^+ (x)=\Frac{1}{x}\left\{\nu-\Frac{a+1}{2}+\sqrt{\left(\nu-\Frac{a+1}{2}\right)^2+x^2}\right\}.
\end{equation}

As we will see, for establishing the bounds on the function ratios it is important that the 
functions $\hat{\gamma}_{a,\nu}^{\pm}(x)$  determining the nullclines 
$\gamma_{t,a,\nu}(x)=\hat{\gamma}_{a,\nu}^{\pm}(x)$ are monotonic. It is easy to prove that:

\begin{lemma} The following monotonicity properties hold:

$\hat{\gamma}_{a,\nu}^+ (x)$ is strictly increasing if $a\le -1$ and strictly decreasing if $a\ge1$.

$\hat{\gamma}_{a,\nu}^- (x)$ is strictly decreasing if $a\le -1$ and strictly increasing if $a\ge 1$.

$\hat{\gamma}_{0,\nu}^\pm (x)$ is strictly decreasing if $\nu> 1/2$, strictly increasing if $\nu< 1/2$
and  constant if $\nu=1/2$.

For $a\in(-1,0)\cup (0,1)$, let
$$
x_e=-\Frac{\sqrt{1-a^2}}{a}\left(\nu-\Frac{a+1}{2}\right).
$$
If $x_e>0$ (respectively $x_e<0$) then $\hat{\gamma}_{a,\nu}^+ (x)$ 
(respectively $\hat{\gamma}_{a,\nu}^- (x)$)
has a relative extremum at $x_e$ (respectively $-x_e$), and it is a minimum (respectively maximum) 
if $a<0$ and a 
maximum (respectively minimum) if $a>0$.

\end{lemma}

We consider next the particular and more important cases $t=0$ and $t=1$ (that is, 
$\gamma_{0,a,\nu}(x)=x^{a} I_{\nu-1}(x)/I_{\nu}(x)$ and 
$\gamma_{1,a,\nu}(x)=-x^{a} K_{\nu-1}(x)/K_{\nu}(x)$) and prove the monotonicity of both functions; later we consider
the general case $t\in (-1,1]$.

\begin{theorem}
\label{moik}
The following monotonicity properties and bounds hold:
\begin{enumerate}
\item{Properties of $\gamma_{0,a,\nu}(x)$:}
\begin{enumerate}
\item{If $|a|>1$} and $\nu\ge 0$, $a\gamma_{0,a,\nu}(x)$ is decreasing and 
$$a(\Phi_{0,\nu}(x)-\lambda^+_{a,\nu}(x))>0$$.
\item{If} $\nu\ge 1/2$,  $\gamma_{0,0,\nu}(x)$ is decreasing and $\Phi_{0,\nu}(x)>\lambda^+_{0,\nu}(x)$.
\end{enumerate}
\item{Properties of $\gamma_{1,a,\nu}(x)$:}
\begin{enumerate}
\item{If} $|a|>1$ then for all $\nu\in {\mathbb R}$, $a\gamma_{1,a,\nu}(x)$ is increasing and 
$$a(\Phi_{1,\nu}(x)+1/\lambda^+_{a,\nu}(x))>0$$.
\item{If} $\nu>1/2$,  $\gamma_{1,0,\nu}(x)$ is decreasing and $\Phi_{1,\nu}(x)<-1/\lambda^+_{0,\nu}(x)$. 

$\Phi_{1,1/2}(x)=-1/\lambda^+_{0,1/2}(x)=-1$.
\end{enumerate}
\end{enumerate}
\end{theorem}

\begin{proof}
We don't give a detailed proof for all the cases, but all the results follow from similar arguments. 

For $t=0$ (modified Bessel function of the first kind), the monotonicity property as $x\rightarrow 0^+$
of  $\gamma_{0,a,\nu}(x)$ (unique for this solutions) is the main input. Then, for instance, in the case
$a\le -1$, $\hat{\gamma}_{a,\nu}^+(x)>0$ is increasing and $\gamma_{0,a,\nu} (x)$ is such that 
$\gamma_{0,a,\nu} ' (0^+)>0$ for $\nu\ge 0$ (see Appendix, Eq. (\ref{seriesI}); 
then necessarily $0<\gamma_{t,a,\nu}(0^+)<\hat{\gamma}_{a,\nu}^+(0^+)$ 
(see (\ref{rca2})) and  
the fact that $\hat{\gamma}_{a,\nu}^+(x)$ is increasing implies that $0<\gamma_{0,a,\nu}(x)<
\hat{\gamma}_{a,\nu}^+(x)$ for all $x>0$. This would prove the result 1.a for $a\le -1$. Similarly for the rest of 1.

For $t=1$ (modified Bessel function of the second kind), the monotonicity property as $x\rightarrow +\infty$
of  $\gamma_{1,a,\nu}(x)$ (unique for this solutions) is the main input. Take for instance the case 
$a\le -1$, when have that $\gamma_{1,a,\nu} ' (+\infty)<0$, which implies, because  $\gamma_{1,a,\nu}(x)<0$ 
(see (\ref{rca2})) that $\gamma_{1,a,\nu}(+\infty)<\hat{\gamma}_{a,\nu}^-(+\infty)$  and  
the fact that $\hat{\gamma}_{a,\nu}^-(x)$ is decreasing implies that $\gamma_{1,a,\nu}(x)<
\hat{\gamma}_{a,\nu}^-(x)<0$ for all $x>0$.

For more detailed proofs, in particular for the cases $a=0,\pm1$,  we refer to \cite{Segura:2011:BFR,Ruiz:2016:ANT}.
\end{proof}

\begin{remark} 
The bounds for $|a|>1$ are weaker
than those for $|a|=1$ in its range of validity. 
\end{remark}

\begin{remark}
\label{moik2}
The range of validity of the previous theorem for the case 1.a extends to $\nu\ge -1$ when $a=-1$ because for this particular
case $\gamma_{0,a,\nu} ' (0^+)>0$ for $\nu\ge -1$ (see Appendix, Eq. (\ref{serin})).
\end{remark}

\begin{remark}
\label{rema}
The bound $\Phi_{0,\nu}(x)>\lambda^+_{0,\nu}(x)$ for $\nu >1/2$ implies, because $\lambda^+_{a,\nu}(x)$ decreases
as a function of $a$, that $\phi_{0,\nu}(x)>\lambda^+_{a,\nu}(x)$ for all $a\ge 0$ and $\nu>1/2$ (which 
implies that $\gamma_{0,a,\nu}(x)$ is monotonically decreasing for $a\in (0,1)$ too). In fact,
the range of validity as a function of $\nu$ increases as $a$ increases from $a=0$ ($\nu>1/2$) to $a=1$ ($\nu>0$). 

Similarly, we have that if $\nu\ge 1/2$ then $\Phi_{1,\nu}(x)<-1/\lambda^+_{a,\nu}(x)$, $a\le 0$, with the range of
validity increasing as $a$ decreases. This implies that $\gamma_{1,a,\nu}(x)$ is also monotonically decreasing 
for $a\in (-1,0)$.

\end{remark}

Now we turn to the general case $t\in (-1,1]$.
We notice that, because as $x\rightarrow +\infty$ the function $I_\nu (x)$ is exponentially increasing while
 $K_\nu (x)$ is exponentially decreasing we have that, for all real $\nu$ and $t\neq 1$,
\begin{equation}
\label{lim1}
\Phi_{t,\nu} (x)\sim \Frac{I_{\nu-1} (x)}{I_{\nu}(x)},\,x\rightarrow +\infty.
\end{equation}
On the other hand, if $\nu\ge 0$ then $I_\nu (x)$ is regular at the origin, while $K_\nu (0^+)=+\infty$ , and 
therefore, for $t\neq 0$ and $\nu\ge 1$,
\begin{equation}
\label{lim2}
\Phi_{t,\nu} (x)\sim -\Frac{K_{\nu-1} (x)}{K_{\nu}(x)},\,x\rightarrow 0^+
\end{equation}

In other words, the behaviour of the solution $\Phi_{1,\nu}(x)=-K_{\nu-1}(x)/K_{\nu}(x)$ 
is unique as $x\rightarrow +\infty$ while
as $x\rightarrow 0$ it is the behavior of $\Phi_{0,\nu}(x)=I_{\nu-1}(x)/I_{\nu}(x)$ 
which is unique. From this information the next result follows, which will be used to prove
 that the monotonicity properties for first and second kind Bessel functions are unique (Theorem \ref{remo}).

\begin{lemma}
\label{region}
Let $\nu\ge 0$ and $D=\{(x,y):\,x>0,\,\gamma_{1,a,\nu}(x)<y<\gamma_{0,a,\nu}(x)\}$. Then 
$\gamma_{t,a,\nu}(x)$ for $t\in (0,1)$ correspond to regular solutions which are inside $D$, while for 
$t\in (-1,0)$ they have a vertical asymptote at $x_*>0$ and their graph is outside $D$. 
\end{lemma}

\begin{proof}
In the first place we notice that the existence and unicity conditions for the solutions of the
Riccati equation are fulfilled and that, therefore, given a point $(x,y)$, $x>0$, there is only one value of
$t$ such that $\gamma_{t,a,\nu}(x)=y$. Therefore, the integral lines can not cross.

Now, taking into account (\ref{lim1}) we know the graph of $\gamma_{t,a,\nu}(x)$
approaches the graph of $\gamma_{0,a,\nu}(x)$ as $x\rightarrow +\infty$ and, on the other hand, it is easy to check that 
$\gamma_{t,a,\nu}(0^+)<\gamma_{0,a,\nu}(0^+)$, $t\neq 0$
(see (\ref{lim2})) for $\nu\ge 1$, and use the series given in the Appendix for $0\le \nu<1$). 

For the case $t\in (-1,0)$, and because $I_{\nu}(x)/K_{\nu}(x)$ increases monotonically from $0$ to $+\infty$
 in $(0,+\infty)$, there exists a single $x_{\infty}>0$ such that $I_{\nu}(x_{\infty})/K_{\nu}(x_{\infty})=
-\tan(\pi t/2)$. Therefore the denominator of (\ref{gene}) is zero at $x_{\infty}$, where the function
has a vertical asymptote. Because the solution tends to $\gamma_{0,a,\nu}(x)>0$ as $x\rightarrow +\infty$, 
then $\gamma_{t,a,\nu}(x)>\gamma_{0,a,\nu}(x)$ for $x>x_{\infty}$ and $\gamma_{t,a,\nu}(x)<\gamma_{1,a,\nu}(x)$ for 
$0<x<x_{\infty}$, and therefore the graph of the solution is outside $D$.

On the other hand, if $t\in (0,1)$ the denominator of (\ref{gene}) is always positive and $\gamma_{t,a,\nu}(x)$ is continuous and 
its graph lies below the graph of $\gamma_{0,a,\nu}(x)$, and therefore is inside the region $D$. 
\end{proof}

\begin{theorem}
\label{remo}
If $\nu\ge 1$, there are no other regular and strictly monotonic solutions of (\ref{rca2}) other than $\gamma_{0,a,\nu}(x)$ and 
$\gamma_{1,a,\nu}(x)$.

$\gamma_{0,a,\nu}(x)$ is strictly monotonic except when $a\in (-1,0)$.

$\gamma_{1,a,\nu}(x)$ is strictly monotonic except when $a\in (0,1)$.

\end{theorem}

\begin{proof}
As discussed before, the solutions with $t\in (-1,0)$ have a discontinuity and therefore, the only thing left to 
prove is that the solutions with $t\in (0,1)$ are not monotonic and the particular cases for $t=0$ and $t=1$ 
give monotonic solutions if $a\notin (0,1)$ and $a\notin (-1,0)$ respectively.

We first observe that Theorem \ref{moik} implies that one or two of the nullclines are inside the region $D$, namely, the graph
of  $\hat{\gamma}_{a,\nu}^+ (x)$ is inside $D$ if $a\ge 0$ while  $\hat{\gamma}_{a,\nu}^- (x)$ is inside $D$ if $a\le 0$.

Now, because of (\ref{lim1}) and (\ref{lim2}) the graph of $\gamma_{t,a,\nu}(x)$, for any $t\in (0,1)$ tends to
the upper boundary of $D$ as $x\rightarrow +\infty$ and to the lower boundary as $x\rightarrow 0^+$. Therefore, it
crosses the nullcline(s) inside $D$. More specifically, there is a local maximum if $a\ge 0$ because 
$\hat{\gamma}_{a,\nu}^+ (x)$ is inside $D$ and a minimum if $a\ge 0$ because 
$\hat{\gamma}_{a,\nu}^- (x)$ is inside $D$.

For $a\in (0,1)$ the solution $\gamma_{0,a,\nu}(x)$ keeps being monotonic (see Remark \ref{rema}), but not 
$\gamma_{1,a,\nu}(x)$ because the derivative changes sign as can be checked by considering the expansions
as $x\rightarrow 0^+$ and $x\rightarrow +\infty$ of the Appendix; $\gamma_{0,a,\nu}(x)$ has a maximum in this case. 
The rest of solutions, can not be regular and 
monotonic, by the same arguments as before. The same can be said for $a\in (-1,0)$, changing the roles of
$\gamma_{0,a,\nu}(x)$ and $\gamma_{1,a,\nu}(x)$; $\gamma_{0,a,\nu}(x)$ has a minimum in this case.
\end{proof}

From the bounds from the Riccati equations and the use of the recurrence relation, most of the know Amos-type inequalities of the form $(\alpha+\sqrt{\beta^2+x^2})/x$ can be established (see 
\cite{Segura:2011:BFR,Ruiz:2016:ANT}), with the exception of the
Simpson-Spector bound \cite{Simpson:1984:SMR}, which follows from arguments similar but not identical to 
the ones considered here for the Riccati 
equations. We will not be exhaustive in the description of these bounds, and we refer to 
\cite{Hornik:2013:ABF} for a systematic analysis of Amos-type bounds. Here we concentrate on the bounds that 
can be extracted from the qualitative analysis of first order differential equations. A way to extend the analysis
was considered in \cite{Ruiz:2016:ANT} by iteration of the Riccati equations, and we explore an in section 
\ref{beyond}
alternative 
possibility by considering the differential equation satisfied by double ratios, similar to that described in 
\cite{Segura:2020:UVS} for Parabolic Cylinder functions.

We end this section with an analysis of the monotonicity properties and bounds for the ratio 
the monotonicity properties discussed so far.

\subsection{Properties for the product $I_{\nu}(x)K_{\nu}(x)$}

We notice that, using the Wronskian relation  \cite[10.28.2]{Olver:2010:BF}
\begin{equation}
\label{wro}
K_{\nu+1}(z)I_{\nu}(z)+K_{\nu}(z)I_{\nu+1}(z)=1/z,
\end{equation}
 and the recurrence relation (\ref{TTRR}) we have 
$$
K_{\nu-1}(z)I_{\nu}(z)+K_{\nu}(z)I_{\nu-1}(z)=1/z,
$$
and then we obtain the following relation with the product 
$I_{\nu}(x) K_{\nu}(x)$:
\footnote{Considering the difference-differential relation (\ref{DDE}) we have that 
$\Frac{I_{\nu-1}(x)}{I_{\nu}(x)}+\Frac{K_{\nu-1}(x)}{K_{\nu}(x)}=
\Frac{d}{dx}\log\left(\Frac{I_{\nu}(x)}{K_{\nu}(x)}\right)$ and therefore the properties
we will establish for the product have a direct counterpart for the logarithmic derivative of the
ratio.
}
$$
\Frac{I_{\nu-1}(x)}{I_{\nu}(x)}+\Frac{K_{\nu-1}(x)}{K_{\nu}(x)}=\Frac{1}{xK_{\nu}(x) I_{\nu}(x)}.
$$
Then, using our previous notation
\begin{equation}
\label{prodre}
K_{\nu}(x)I_{\nu}(x)=\Frac{1}{x(\Phi_{0,\nu}(x)-\Phi_{1,\nu}(x))}=\Frac{1}{\gamma_{0,-1,\nu}(x)-\gamma_{1,-1,\nu}(x)}.
\end{equation}

Now we notice that Theorem \ref{moik} and Remark \ref{moik2} estate that both $\gamma_{0,-1,\nu}(x)$ 
and $-\gamma_{1,-1,\nu}(x)$ are increasing functions if $\nu\ge -1$, and that this proves that 
$I_{\nu}(x)K_{\nu}(x)$ is decreasing for $\nu\ge -1$. This enlarges the range of validity of the result 
proved in \cite{Penfold:2007:MOS}, which was later extended to $\nu\ge -1/2$ in \cite{Baricz:2009:OAP}. Here we 
have just proved this 
result in a very straightforward way and in the larger range $\nu\ge -1$. We also prove next that $xI_{\nu}(x)K_{\nu}(x)$
is increasing for $\nu\ge 1/2$. We collect both results in a single theorem:

\begin{theorem}
Let $f_{\lambda,\nu}(x)=x^{\lambda}I_{\nu}(x)K_{\nu}(x)$, then
\begin{enumerate}
\item{If} $\lambda\le 0$ and  $\nu\ge -1$ $f_{\lambda,\nu}(x)$ is strictly decreasing for $x>0$.
\item{If} $\lambda\ge 1$ and  $\nu\ge 1/2$ $f_{\lambda,\nu}(x)$ is strictly increasing for $x>0$.
\end{enumerate}
\end{theorem}

\begin{proof}
We only need to prove this result for $\lambda=0,1$; for the rest of values it follows immediately.

For $\lambda=0$, as commented before, the relation (\ref{prodre}) and the fact that both 
$\gamma_{0,-1,\nu}(x)$ 
and $-\gamma_{1,-1,\nu}(x)$ are strictly increasing functions proves the result.

For $\lambda=1$ we have
$$
f_{1,\nu}(x)=xI_{\nu}(x)K_{\nu}(x)=\Frac{1}{\Phi_{0,\nu}(x)-\Phi_{1,\nu}(x)},
$$
and we need to prove that $\Phi'_{0,\nu}(x)-\Phi'_{1,\nu}(x)<0$ if $\nu\ge 1/2$.

Using (\ref{Riccati}), which is satisfied by $\Phi_{0,\nu}(x)$ and $\Phi_{1,\nu}(x)$, we have
$$
\Phi'_{0,\nu}(x)-\Phi'_{1,\nu}(x)=(\Phi_{0,\nu}(x)-\Phi_{1,\nu}(x))
\left[\Frac{2\nu-1}{x}-\left(\Phi_{0,\nu}(x)+\Phi_{1,\nu}(x)\right)\right],
$$
and because $\Phi_{0,\nu}(x)>0$ and $\Phi_{1,\nu}(x)<0$ we only need to prove that
$\Phi_{0,\nu}(x)+\Phi_{1,\nu}(x)>\Frac{2\nu-1}{x}$ if $\nu\ge 1/2$, which is easy to check 
by using some of the bounds of Theorem \ref{moik}. Namely, we use that $\Phi_{0,\nu}(x)>\lambda_{0,\nu}^+ (x)$ 
if $\nu\ge 1/2$ and that $\Phi_{1,\nu}(x)>-1/\lambda_{1,\nu}^+ (x)$ for all real $\nu$. Then, for $\nu\ge 1/2$:
$$
\begin{array}{ll}
\Phi_{0,\nu}(x)+\Phi_{1,\nu}(x)&>\Frac{\nu+\sqrt{\nu^2+x^2}}{x}-\Frac{x}{\nu-1+\sqrt{(\nu-1)^2+x^2}}\\
& \\
& = \Frac{2\nu-1}{x}+\sqrt{\nu^2+x^2}-\sqrt{(\nu-1)^2+x^2}\ge\Frac{2\nu-1}{x}
\end{array}
$$
\end{proof}

\begin{remark}
For $\lambda\in (0,1)$ and $\nu>0$, $f_{\lambda,\nu}(x)$ is not monotonic, as can be easily checked using
 (\ref{p1}) and (\ref{p2}).

For $\lambda>1$ the range for which $f_{\lambda,\nu}(x)$ increases becomes larger as 
$\lambda$ increases; for $\lambda=2$
 Theorem \ref{moik} guarantees that the result is valid at least for $\nu\ge 0$. We don't analyze here these
further details. 
\end{remark}

Using (\ref{prodre}) and the bounds for the ratios of Bessel functions, sharp bounds for the product can be
established. We will not be exhaustive in this discussion, as the bounds can be straightforwardly derived. We just
give two of these bounds, which are obtained from Theorem \ref{moik} (more bounds are available from this same theorem).

\begin{theorem} The following two bounds hold:
$$
I_{\nu}(x)K_{\nu}(x)<\Frac{1}{2\sqrt{(\nu-1/2)^2+x^2}},\,\nu\ge 1/2,
$$
$$
I_{\nu}(x)K_{\nu}(x)>\Frac{1}{1+\sqrt{\nu^2+x^2}+\sqrt{(\nu-1)^2+x^2}},\,\nu\ge -1
$$
\end{theorem}
\begin{proof}
For the upper bound use (\ref{prodre}) and that $\Phi_{0,\nu}(x)>\lambda^+_{0,\nu}(x)$ and 
$\Phi_{1,\nu}(x)<-1/\lambda^+_{0,\nu}(x)$, $\nu\ge 1/2$. For the lower bound use  
$\Phi_{0,\nu}(x)<\lambda^+_{-1,\nu}(x)$, $\nu\ge -1$, and $\Phi_{1,\nu}(x)>-1/\lambda^+_{1,\nu}(x)$, 
$\nu\in {\mathbb R}$.
\end{proof}

These two bounds (as all the bounds that can be extracted from Theorem \ref{moik}) are sharp as 
$x\rightarrow +\infty$. They are not sharp, however, as $x\rightarrow 0^+$. This is in contrast
 with the bounds in \cite[Thm. 2]{Baricz:2016:BFT}, which are sharp as $x\rightarrow 0^+$ and $\nu>0$
 but not as $x\rightarrow +\infty$. This is as expected, because we are using bounds for the ratios 
which are sharp as $x\rightarrow +\infty$ but not as $x\rightarrow 0^+$; however, upper and lower 
bounds for the 
ratios which are sharp in both limits are available (see for instance 
\cite{Ruiz:2016:ANT}), and from there it is straightforward
to obtain sharp bounds for the product. In particular, the bounds from the iteration of the Riccati equation
given in \cite{Ruiz:2016:ANT} are sharp in both limits. We don't give here such bounds for 
the product explicitly, which are straightforward applications
of previous results, but we will obtain later 
a new very sharp trigonometric bound which is very accurate in the three limits 
$x\rightarrow 0^+$, $x\rightarrow +\infty$
and $\nu\rightarrow +\infty$.

\section{Very sharp trigonometric bounds}
\label{beyond}

As in \cite{Segura:2020:UVS}, we will study the monotonicity properties of the double ratios 
\begin{equation}
W_{\nu}(x)=\Phi_{\nu}(x)/\Phi_{\nu+1}(x), 
\end{equation}
and we will establish very sharp trigonometric bounds from these monotonicity properties.

That the double ratios are monotonic both for the first and second kind modified Bessel functions, has
been separately shown in two different papers by different methods 
\cite{Yang:2018:MAC,Yang:2017:TMA}. First, in \cite{Yang:2017:TMA} it was proved that 
$K_{u}(x)K_{v}(x)/K_{(u+v)/2}(x)^2$ is strictly decreasing for $x>0$ and real $u,v$; integral representation
for the product and ratios of modified Bessel functions of the second kind were considered in this analysis. 
Later, in \cite{Yang:2018:MAC} 
it was shown that the ratio $I_{u}(x)I_{v}(x)/I_{(u+v)/2}(x)^2$, $\min\{u,v\}>-2$, $u+v>-2$, $u,v\neq -1$ 
is strictly increasing for $x>0$, using the Frobenius series for the Bessel functions.  
Here, we give a more restricted  version of these properties ($|u-v|=2$), 
but we do this in a single analysis for the first and second kind functions, 
we prove that such monotonicity
properties are unique for these two solutions and we obtain bounds for the ratios and products
 that are sharper and of a different type to those obtained with previous analysis.

Using (\ref{TTRR}) we have
\begin{equation}
W_{\nu}(x)=\Phi_{\nu}(x)\left(\Phi_{\nu}(x)-\Frac{2\nu}{x}\right)
\end{equation}

And in terms of 
$$
\psi_\nu (x)=x\Phi_\nu (x)-\nu=x\Frac{\cali'_{\nu}(x)}{\cali_\nu (x)},
$$
which satisfies
\begin{equation}
\label{psip}
x\psi'_{\nu}(x)=\nu^2+x^2-\psi_\nu (x)^2,
\end{equation}
we have that
\begin{equation}
\label{W}
W_{\nu}(x)=\Frac{\psi_{\nu}(x)^2-\nu^2}{x^2},
\end{equation}	
and differentiating (\ref{W})
\begin{equation}
\label{odew}
W'_{\nu}(x)=-\Frac{2}{x^3}\left(\psi_\nu (x)^3+\psi_\nu (x)^2-(\nu^2+x^2)\psi_\nu (x)-\nu^2\right).
\end{equation}

Next we will analyze the qualitative properties of the solutions of the system of equations 
(\ref{odew})--(\ref{W}), and from these we will obtain very sharp trigonometric bounds. 
Notice that (\ref{odew}) has been obtained by differentiating (\ref{W}) and using (\ref{psip}) and
that, conversely, differentiating (\ref{W}) and using (\ref{odew}) we obtain that
the possible differentiable $\psi_\nu$-solutions of the system (\ref{odew})--(\ref{W}) are the trivial solution 
$\psi_{\nu}(x)=0$ and the solutions of (\ref{psip}), with general solution given by (\ref{gene}).
For obvious reasons (the objective is to find properties for modified Bessel functions) we are only considering the
latter solutions, in which case the solutions $W_{\nu}(x)$ are
\begin{equation}
\begin{array}{ll}
W_{\nu}(x)&\equiv W_{t,\nu}(x)=
\Phi_{t,\nu}(x)/\Phi_{t,\nu+1}(x)\\
&\\
&\hspace*{-0.5cm} =\Frac{(\cos\alpha I_{\nu-1}(x)-\sin\alpha K_{\nu-1}(x))(\cos\alpha I_{\nu+1}(x)-\sin\alpha K_{\nu+1}(x))}{(\cos\alpha I_{\nu}(x)+\sin\alpha K_{\nu}(x))^2},
\end{array}
\end{equation}
where $\alpha=\pi t/2$, $t\in (-1,1]$.

\begin{remark}
We notice that, because of (\ref{W}), 
$
W_{t,\nu}(x)>-\Frac{\nu^2}{x^2}
$
for any real $t$.
\end{remark}

From now on, we will drop the notation $W_{t,\nu}(x)$ in favor or $W_{\nu}(x)$. We will recover it in 
Theorem \ref{doublemon}.

For analyzing the qualitative properties of the solutions of the system (\ref{odew})--(\ref{W}) we need to
analyze the nullclines of (\ref{odew}), which determine the monotonicity properties of the solutions.

\subsection{Properties of the nullclines}

For proving the monotonicity properties and bounds for the double ratio $W_{\nu}(x)$, we need to 
analyze the nullclines in terms of the values of $\psi_{\nu}(x)$ which make the right-hand of Eq. (\ref{odew}) 
zero and then to study the corresponding values of $W_{\nu}(x)$ and their monotonicity. We first analyze in Lemma
 \ref{cubicl} the nullclines in terms of the values of $\psi_{\nu}(x)$; after this, the properties for the corresponding 
values of $W_{\nu}(x)$ are analyzed in Lemmas \ref{mainlemma} and \ref{mainlemma2}. 
Once these lemmas are proved, the main
 results can be estated. 

\begin{lemma}
\label{cubicl}
The cubic equation 
\begin{equation}
\label{cubic}
\lambda_{\nu}(x)^3+\lambda_{\nu}(x)^2-(\nu^2+x^2)\lambda_{\nu}(x)-\nu^2=0,\,x\neq 0,\,\nu\neq 0
\end{equation}
has three distinct real roots
$$
\lambda_{\nu}^{(K)}(x)<\lambda_{\nu}^{(O)}(x)<\lambda_\nu^{(I)}(x)
$$
such that $\lambda_{\nu}^{(K)}(x)<-|\nu|$, $\lambda_\nu^{(O)}(x)\in (-|\nu|,0)$ and 
$\lambda_\nu^{(I)}(x)>|\nu|$.

These solutions can be written
\begin{equation}
\label{solu3}
\lambda_{\nu}(x)=\Frac{2}{3}g_{\nu}(x)\cos\left(\Frac{1}{3}\arccos\left(\Frac{18\nu^2-9x^2-2}{2g_{\nu}(x)^3}
\right)+\alpha\right)-\Frac{1}{3},
\end{equation}
where 
$$
g_{\nu}(x)=\sqrt{3(\nu^2+x^2)+1}.
$$
$\alpha=0$ for $\lambda_\nu^{(I)}(x)$, $\alpha=2\pi/3$ for $\lambda_\nu^{(K)}(x)$ and 
$\alpha=-2\pi/3$ for $\lambda_\nu^{(O)}(x)$.

The three solutions are even functions of $x$. For $x>0$ 
$\lambda_{\nu}^{(I)}(x)$ and $\lambda_{\nu}^{(O)}(x)$ are strictly increasing and 
$\lambda_{\nu}^{(K)}(x)$ strictly decreasing.
\end{lemma}

\begin{proof}
Let $f (\lambda)=\lambda^3+\lambda^2-(\nu^2+x^2)\lambda-\nu^2$, we have
that $f (-\infty)=-\infty$, $f(-|\nu|)=x^2 |\nu|>0$, $f(0)=-\nu^2 <0$, $f(\nu)=-x^2 |\nu|>0$, $f(+\infty)=+\infty$.
 Therefore, by Bolzano's theorem, there is for any $x\neq 0$ one root in $(-\infty,-|\nu|)$ 
($\lambda_{\nu}^{(K)}(x)$), a second root in $(-|\nu|,0)$ ($\lambda_{\nu}^{(O)}(x)$) and finally a root
in $(|\nu|,+\infty)$ ($\lambda_{\nu}^{(I)}(x)$).

For solving the equation, we transform the cubic $f(\lambda)=0$ to depressed form with 
the change $\lambda=\mu-\frac13$ and we have $\mu^3-p\mu-q=0$ with $p=\nu^2+x^2+\frac13$ and 
$q=\frac23 \nu^2 -\frac13 x^2-\frac{2}{27}$ and using the well know trigonometric formula for the solution of
a depressed cubic:
$$
\mu=2\sqrt{p/3}\cos\left(\frac13\arccos\left(\Frac{3q}{2p\sqrt{p/3}}\right)+\alpha\right),\alpha=0,\pm 2\pi/3,
$$
from where we have that the solutions have the form (\ref{solu3}). The solutions are differentiable when 
the absolute value of the argument of the $\arccos$ is smaller than $1$, which is equivalent to saying
that the discriminant $\Delta=-(4p^3+27 q^2)$ is positive, and we have
$$
\Delta =4x^6+(12\nu^2+1)x^4+(12\nu^4+20\nu^2)x^2+4\nu^2(\nu^2-1)^2,
$$
which is positive for $x\neq 0$ (see remark \ref{equiscero}).

Now we expand as $x\rightarrow +\infty$ for the three values $\alpha=0,\pm 2\pi/3$ and we get
\begin{equation}
\label{asil}
\begin{array}{ll}
\alpha=0, &\lambda_{\nu}^{(I)}(x)=x-\Frac{1}{2}+\Frac{\nu^2+1/2}{2x}+\Frac{\nu^2}{2x^2}+{\cal O}(x^{-3}),\\
\alpha=2\pi/3, &\lambda_{\nu}^{(K)}(x)=-x-\Frac{1}{2}-\Frac{\nu^2+1/2}{2x}+\Frac{\nu^2}{2x^2}+{\cal O}(x^{-3}),\\
\alpha=-2\pi/3, &\lambda_{\nu}^{(O)}(x)=-\Frac{\nu^2}{x^2}+\Frac{\nu^4}{x^4}+{\cal O}(x^{-6}),
\end{array}
\end{equation}
which shows that the ordering $\lambda_{\nu}^{(K)}(x)<\lambda_{\nu}^{(O)}(x)<\lambda_{\nu}^{(I)}(x)$ is correct.

That the solutions are even functions is immediate given the symmetries of the equation and the monotonicity follows
by taking the derivative of (\ref{cubic}), from where
\begin{equation}
\label{derl}
\lambda'_{\nu}(x)=\Frac{2x\lambda_{\nu}(x)}{3\lambda_{\nu}(x)^2+2\lambda_{\nu}(x)-(\nu^2+x^2)}.
\end{equation}

Consider now $x>0$. The numerator does not change sign and neither does the denominator, because if the 
denominator was zero for some $x>0$ then $\lambda_{\nu}(x)$ would not be differentiable for this $x$, which
 can not be true.
Then, the sign of the denominator for $x>0$ is equal to its sign as $x\rightarrow +\infty$ which, 
using (\ref{asil}), is positive for $\lambda_{\nu}(x)= \lambda_{\nu}^{(I)}(x)$ 
and $\lambda_{\nu}(x)= \lambda_{\nu}^{(K)}(x)$, while it is negative for 
$\lambda_{\nu}(x)= \lambda_{\nu}^{(O)}(x)$. Now, because $\lambda_{\nu}^{(I)}(x)>0$
 and the other two solutions are negative the monotonicity properties follow.
\end{proof}

\begin{remark}
The notation $\lambda_{\nu}^{(I)}(x)$ is used because this solution will be related to a bound for 
$I_{\nu-1}(x)/I_{\nu}(x)$. Similarly, $\lambda_{\nu}^{(K)}(x)$ is related to $K_{\nu-1}(x)/K_{\nu}(x)$.
\end{remark}

\begin{remark}
\label{nucero}
For $\nu=0$ the solutions are, trivially, $\lambda_{0}^{(I)}(x)=(-1+\sqrt{1+4x^2})/2$, $\lambda_0^{(O)}(x)=0$ and
$\lambda_{0}^{(I)}(x)=(-1-\sqrt{1+4x^2})/2$.
\end{remark}

\begin{remark}
\label{equiscero}
As explained in the previous proof, the solutions $\lambda_{\nu}(x)$ are simple if $x\neq 0$. We will be
interested in the case $x>0$. For $x=0$, the solutions are trivially $\lambda_{\nu}(x)=-1,\pm \nu$.
The identification with the previous notation is given by $\lambda_{\nu}^{(I)}(0)=|\nu|$, 
$\lambda_{\nu}^{(K)}(0)=\min\{-1,-|\nu|\}$, $\lambda_{\nu}^{(O)}(0)=\max\{-1,-|\nu|\}$ (which implies quite 
curious trigonometric identities)
\end{remark}

\begin{lemma}
\label{mainlemma}
We define
\begin{equation}
\label{ws}
\begin{array}{l}
w_{\nu}^{(I)}(x)=(\lambda_{\nu}^{(I)}(x)^2-\nu^2)/x^2,\\
w_{\nu}^{(K)}(x)=(\lambda_{\nu}^{(K)}(x)^2-\nu^2)/x^2,\\
w_{\nu}^{(O)}(x)=(\lambda_{\nu}^{(O)}(x)^2-\nu^2)/x^2.
\end{array}
\end{equation}
Then, for all $x>0$, $w_{\nu}^{(K)}(x)>w_{\nu}^{(I)}(x)>0$, $w_{\nu}^{(O)}(x)<0$. 
\end{lemma}
\begin{proof}
First we observe that the fact that $w_{\nu}^{(I)}(x)>0$ and $w_{\nu}^{(K)}(x)>0$ while $w_{\nu}^{(O)}(x)<0$
is a consequence of the fact that (see lemma \ref{cubicl}), $|\lambda_{\nu}^{(I)}(x)|>|\nu|$ and 
$|\lambda_{\nu}^{(K)}(x)|<|\nu|$ while $|\lambda_{\nu}^{(I)}(x)|>|\nu|$. 

To prove that $w_{\nu}^{(K)}(x)>w_{\nu}^{(I)}(x)$ for $x>0$ we check that $w_{\nu}^{(I)}(x)\neq w_{\nu}^{(I)}(x)$
for all $x>0$ and that $w_{\nu}^{(K)}(x)>w_{\nu}^{(I)}(x)$ for large $x$. 

Indeed, if we had 
$w_{\nu}^{(I)}(x)= w_{\nu}^{(I)}(x)$ this would imply that for such $x>0$  $\lambda_{\nu}^{(I)}(x)=-
\lambda_{\nu}^{(K)}(x)$. Then, for this value of $x$ we have that both $\lambda_{\nu}(x)=\lambda_{\nu}^{(I)}(x)$ and
$-\lambda_{\nu}(x)$ are solutions of
 (\ref{cubic}). Then:
$$
\begin{array}{l}
\lambda_{\nu}(x)^3+\lambda_{\nu}(x)^2-(\nu^2+x^2)\lambda_{\nu}(x)-\nu^2=0,\\
-\lambda_{\nu}(x)^3+\lambda_{\nu}(x)^2+(\nu^2+x^2)\lambda_{\nu}(x)-\nu^2=0,
\end{array}
$$
and adding both equations $\lambda_n (x)^2=\lambda_{\nu}^{(I)}(x)^2=\nu^2$, which does not hold for $x>0$. Therefore
$w_{\nu}^{(I)}(x)\neq w_{\nu}^{(I)}(x)$ for all $x>0$.

Now, expanding as $x\rightarrow +\infty$:
\begin{equation}
\label{asinw}
\begin{array}{l}
w_{\nu}^{(I)}(x)=1-\Frac{1}{x}+\Frac{1}{2x^2}+\Frac{\nu^2-1/4}{2x^3}+{\cal O}(x^{-4}),\\
w_{\nu}^{(K)}(x)=1+\Frac{1}{x}+\Frac{1}{2x^2}-\Frac{\nu^2-1/4}{2x^3}+{\cal O}(x^{-4}),
\end{array}
\end{equation}
and the first two terms suffice to see that $w_{\nu}^{(K)}(x)>w_{\nu}^{(I)}(x)$ for large $x$ and therefore
that this holds for all $x>0$.
\end{proof}

\begin{lemma}
\label{mainlemma2}
For all real $\nu$ and for $x>0$ 
we have $w_{\nu}^{(K)\prime}(x)<0$, $w_{\nu}^{(I)\prime}(x)>0$ and $w_{\nu}^{(O)\prime}(x)>0$, except that
$w_{0}^{(O)}(x)=0$. 
\end{lemma}

\begin{proof}
We take the derivative in $w_{\nu}(x)=(\lambda_{\nu}(x)-\nu^2)/x^2$, where $\lambda_{\nu}(x)$ is any of the solutions
of (\ref{cubic})
$$
w'_{\nu}(x)=\Frac{2}{x^3}\left(\nu^2-\lambda_{\nu}(x)^2+x\lambda_{\nu}(x)\lambda'_{\nu}(x)\right)
$$
and using (\ref{derl}) we have
$$
\Frac{x^3}{2} f_{\nu}(x) w'_{\nu} (x)=
-3\lambda_{\nu}(x)^4 -2\lambda_{\nu}(x)^3+(4\nu^2+4x^2)\lambda_{\nu}(x)^2 +
2\nu^2\lambda_{\nu}(x)-\nu^2(\nu^2+x^2)
$$
where $f_{\nu}(x)=3\lambda_{\nu}(x)^2+2\lambda_{\nu}(x)-(\nu^2+x^2)$ which, as discussed at the end of the
 proof of lemma \ref{mainlemma}, is negative for $\lambda_{\nu}(x)=\lambda_{\nu}^{(O)}(x)$ and positive 
for the other two roots.

Now writing $-\lambda_\nu (x)^4=-\lambda_\nu (x)\lambda_\nu (x)^3$ and using (\ref{cubic}) to eliminate 
$\lambda_{\nu}(x)^3$ we arrive at
\begin{equation}
\label{cubic2}
\begin{array}{l}
\Frac{x^3}{2} f_{\nu}(x) w'_{\nu} (x)=h_{\nu}(x),\\
h_{\nu}(x)=\lambda_\nu (x)^3 +\nu^2\lambda_\nu (x)^2-\nu^2 \lambda_\nu(x)
-\nu^2 (\nu^2 +x^2)
\end{array}
\end{equation}

Now, we are proving that none of the solutions of (\ref{cubic}) are such that $w'_{\nu}(x)=0$ for any $x>0$. 
After we have proved this, we will only need to analyze the sign of $w'_{\nu}(x)$ as $x\rightarrow +\infty$
in order to prove the lemma. In other words, what we need to prove first is that for any 
$x>0$ no $\lambda_\nu (x)$ exists such both right-hand sides of (\ref{cubic}) and (\ref{cubic2}) vanish, that is,
that no $\lambda_\nu (x)$ exists such that
\begin{equation}
\label{sistema}
\begin{array}{l}
h_{\nu}(x)=\lambda_\nu (x)^3 +\nu^2\lambda_\nu (x)^2-\nu^2 \lambda_\nu(x)-\nu^2 (\nu^2 +x^2)=0,\\
\lambda_{\nu}(x)^3+\lambda_{\nu}(x)^2-(\nu^2+x^2)\lambda_{\nu}(x)-\nu^2=0
\end{array}
\end{equation}
for no $x>0$.

We subtract both equations and then
$$
(\nu^2-1)\lambda_{\nu}(x)^2+x^2\lambda_{\nu}(x)-\nu^2(\nu^2+x^2-1)=0.
$$
For $\nu^2=1$ the solution is $\lambda_{\nu}(x)=-1$ for which $h_{\nu}(x)=x^2\neq 0$. For $\nu^2\neq 1$ we solve
the quadratic equation and substitute the solutions in the expression of $h_{\nu}(x)$, yielding
\begin{equation}
\begin{array}{l}
h_{\nu}(x)=\Frac{x^4}{2(\nu^2-1)^3}\left(B\pm\sqrt{\Delta}\right),\\
\\
B=x^2+2\nu^2(\nu^2-1),\,\Delta=x^4+4\nu^2(\nu^2-1)(\nu^2+x^2-1),
\end{array}
\end{equation}
which should be zero. However, this is only possible if $x=0$ (and we are not considering this case) or if
 $\nu=0$, in which case (\ref{sistema}) implies $\lambda_\nu (x)=0=\lambda_{\nu}^{(O)}(x)$ (see remark \ref{nucero}).
Indeed, for $h_{\nu}(x)$ to be zero we need first that $\Delta\ge 0$ and then also that $\Delta-B^2=0$, but
$$
\Delta-B^2=4\nu^2(\nu^2-1)^3
$$
which is different from zero unless $\nu^2=1$ (case already considered) and $\nu=0$, which is the trivial case
$\lambda_\nu (x)=0=\lambda_{\nu}^{(O)}(x)$.

With this, we have demonstrated that $w_{\nu}^{(A)\prime}(x)$, $A=I,K,O$ do not change sign for $x>0$ and
$\nu\neq 0$. Now, from (\ref{asinw}) we see that $w_{\nu}^{(I)\prime}(x)>0$,  $w_{\nu}^{(K)\prime}(x)<0$. 
On the other hand, computing the expansion for $w_{\nu}^{(O)}(x)$ as $x\rightarrow \infty$:
$$
w_{\nu}^{(O)}(x)=-\Frac{\nu^2}{x^2}+{\cal O}(x^{-6}),
$$
and therefore $w_{\nu}^{(O)\prime}(x)>0$
\end{proof}

\subsection{Main results}

\begin{theorem}
\label{lemaprin}
Let us consider the differential equation 
\begin{equation}
W'_{\nu}(x)=-\Frac{2}{x^3}\left(\psi_\nu (x)^3+\psi_\nu (x)^2-(\nu^2+x^2)\psi_\nu (x)-\nu^2\right),
\end{equation}
with 
\begin{equation}
\label{Wp}
W_{\nu}(x)=(\psi_{\nu}(x)^2-\nu^2)/x^2,
\end{equation}
and $x>0$. 
Let $\lambda_{\nu}^{(A)}(x)$ and $w_{\nu}^{(A)}(x)$, 
$A=I,O,K$, be as defined in lemmas 
\ref{cubicl} and \ref{mainlemma}.

The following holds:
\begin{enumerate}
\item{} If $\psi_{\nu}(0^+)>0$, $W_{\nu}(0^+)>0$ and $W'_{\nu}(0^+)>0$ then $W'_{\nu}(x)>0$, $\psi'_{\nu}(x)>0$, 
$|\nu|<\psi_{\nu}(x)<\lambda_{\nu}^{(I)}(x)$ and
$0<W_{\nu}(x)<w_{\nu}^{(I)}(x)$ for all $x>0$.
\item{} If $\psi_\nu (+\infty)<0$, $W_{\nu}(+\infty)>0$ and $W'_{\nu}(+\infty)<0$ then $W'_{\nu}(x)<0$,
 $\psi'_{\nu}(x)<0$,
$\lambda_{\nu}^{(K)}(x)<\psi_{\nu}(x)<-|\nu|$ and
$0<W_{\nu}(x)<w_{\nu}^{(K)}(x)$ for all $x>0$.
\end{enumerate}
\end{theorem}

\begin{proof}
In both cases we have to take into account that
$$
W'_{\nu}(x)=-\Frac{2}{x^3}(\psi_{\nu}(x)-\lambda_{\nu}^{(I)}(x))(\psi_{\nu}(x)-\lambda_{\nu}^{(O)}(x))
(\psi_{\nu}(x)-\lambda_{\nu}^{(K)}(x))
$$
where $\lambda_{\nu}^{(I)}(x)>\lambda_{\nu}^{(O)}(x)>\lambda_{\nu}^{(K)}(x)$ and only $\lambda_{\nu}^{(I)}(x)$
is positive (lemma \ref{cubicl}).

For proving this result we will let  the solution of the differential equation evolve from $x=0^+$ to $+\infty$ in
the first case and from $+\infty$ to $0^+$ in the second case, checking that none of the nullclines (curves 
where $W'_{\nu}(x)=0$) is reached
by the solution and therefore the monotonicity does not change, which in turn implies the
bounds for $\psi_{\nu}(x)$ and $W_{\nu}(x)$. We prove in detail the first case; the second case can be proved in an
 analogous way.

1. Since $\lambda_{\nu}^{(I)}(x)>0$ while $\lambda_{\nu}^{(K)}(x)<\lambda_{\nu}^{(O)}(x)<0$, and because  
$\psi_{\nu}(0^+)>0$ and $W'_{\nu}(0^+)>0$, then $0<\psi_{\nu}(0^+)<\lambda_\nu^{(I)}(0^+)$. 
In addition $W_{\nu}(0^+)>0$
 which means, using  (\ref{Wp}) $\psi_{\nu}(0^+)>|\nu|$ and differentiating (\ref{Wp}), that also 
$\psi'_{\nu}(0^+)>0$.

As a first step, we prove that if there exists $x_{\nu}$ such that $\psi_{\nu}(x_\nu)=|\nu|$, 
$\psi_{\nu}(x_\nu)>\nu$ in $(0,x_\nu$),
 there there exists 
a value $0<x_e<x_\nu$ such that $\psi_\nu (x_e)=\lambda_{\nu}^{(I)} (x_e)$. Indeed, because $\psi_{\nu}(0^+)>|\nu|$ by
 Rolle's theorem there exists $0<x_m<x_\nu$ such that $\psi'_{\nu} (x_m)=0$, and because differentiating (\ref{Wp})
$2\psi_\nu (x)\psi'_{\nu}(x)=x^2W'_\nu (x)+2xW_\nu (x)$, then $W_\nu (x_m)W'_\nu(x_m)<0$ which means, because
$W_{\nu}(0^+)>0$ and $W'_{\nu}(0^+)>0$, that there must exist $0<x_e<x_m$ such that $W'_\nu (x_e)=0$; now, since 
$\psi_\nu (x_e)>|\nu |$ necessarily $\psi_\nu(x_e)=\lambda_\nu^{(I)}(x_e)$. 
This proves that $\psi_{\nu}(x)>|\nu|$
and therefore $W_\nu (x)>0$ as long as $x<x_e$ and that if the solution crosses a nullcline, the first
nullcline which is crossed must be the one corresponding to $\lambda_\nu^{(I)}(x)$.

Now, we consider the potential interval $(0,x_e)$, where $\phi_\nu (x)>\nu$, $W_\nu (x)>0$, and prove that such
value $x_e$ for which $\psi_\nu(x_e)=\lambda_\nu^{(I)}(x_e)$ does
not exist.

Because $\psi_{\nu}(0^+)<\lambda_\nu^{(I)} (0^+)$ we have that $W_\nu (0^+)<w_{\nu}^{(I)}(0^+)$ 
 and therefore the curve $y=W_{\nu}(x)$ lies below the curve 
$y=w_{\nu}^{(I)}(x)$ for $x=0^+$. 
But then it is not possible that the
curve $y=W_{\nu}(x)$ interesects the curve $y=w_{\nu}^{(I)}(x)$, because at the possible 
intersection point $W_\nu(x_e)=w_\nu^{(I)}(x_e)$, and because $\psi_\nu(x)>|\nu|$ in $(0,x_e)$,  we then have
$\psi_{\nu}(x)=\lambda_{\nu}^{(I)}(x)$ and therefore $W'_\nu (x_e)=0$. This is contradictory with the facts 
that $w_{\nu}^{(I)}(x)$ is increasing (see lemma \ref{mainlemma2}) and that the curve $y=W_{\nu}(x)$ lies 
below the curve 
$y=w_{\nu}^{(I)}(x)$.

Then, because the crossing point $x_e$ does not exist we have proved that $\nu<\psi_{\nu}(x)<\lambda_{\nu}^{(I)}(x)$ 
for all $x$, from where the rest of results follow. 

\end{proof}

We now establish the bounds for first and second kind modified Bessel function ratios

\begin{theorem}
The following results hold for $\nu\ge 0$, $x>0$:

\begin{enumerate}
\item{} $\Psi_{\nu}^{(I)}(x)=x\Frac{I'_{\nu}(x)}{I_{\nu}(x)}$ and $
W_{\nu}^{(I)}(x)=\Frac{I_{\nu-1}(x)I_{\nu+1}(x)}{I_{\nu}(x)^2}$ are strictly increasing functions.
\item{} $\nu<\Psi_{\nu}^{(I)}(x)<\lambda_\nu^{(I)}(x)$ and $0<W_\nu^{(I)}(x)<w_\nu^{(I)}(x)$.
\item{} $\Psi_{\nu}^{(K)}(x)=x\Frac{K'_{\nu}(x)}{K_{\nu}(x)}$ and $
W_{\nu}^{(K)}(x)=\Frac{K_{\nu-1}(x)K_{\nu+1}(x)}{K_{\nu}(x)^2}$ are strictly decreasing functions.
\item{} $\lambda_\nu^{(K)}(x)<\Psi_{\nu}^{(K)}(x)<-\nu$ and $0<W_\nu^{(K)}(x)<w_\nu^{(K)}(x)$.
\end{enumerate}

\end{theorem}

\begin{proof}
Let $\nu\ge 0$ and $x>0$. Choosing $\psi_{\nu}(x)$ as
$$
\psi_{\nu}(x)=x\Frac{I_{\nu-1}(x)}{I_\nu (x)}-\nu=x\Frac{I'_{\nu}(x)}{I_{\nu}(x)}
$$
the hypothesis (1) of lemma \ref{lemaprin} are fulfilled, and choosing
$$
\psi_{\nu}(x)=-x\Frac{K_{\nu-1}(x)}{K_\nu (x)}-\nu=x\Frac{K'_{\nu}(x)}{K_{\nu}(x)}
$$
the hypothesis (2) of the same lemma are satisfied.
\end{proof}

\begin{corollary}
\label{cotastrig}
For $x>0$ and $\nu\ge 0$ the following holds:
\begin{equation}
\begin{array}{l}
\Frac{I_{\nu-1}(x)}{I_\nu (x)}<\Frac{2g_\nu (x)}{3x}\cos\left(\Frac{1}{3}\arccos\left(
\Frac{h_{\nu}(x)}{g_\nu (x)^3}
\right)\right)+\Frac{\nu-1/3}{x},\\
\Frac{K_{\nu-1}(x)}{K_\nu (x)}<\Frac{2g_\nu (x)}{3x}\cos\left(\Frac{1}{3}\arccos\left(

\Frac{h_{\nu}(x)}{g_\nu (x)^3}
\right)-\Frac{\pi}{3}\right)-\Frac{\nu-1/3}{x},
\end{array}
\end{equation}
where $g_\nu (x)=\sqrt{3(\nu^2+x^2)+1}$, $h_{\nu}(x)=9\nu^2-\Frac{9}{2}x^2-1$.
\end{corollary}

Same as happened in section \ref{riccatis}, it turns out that the first and second kind modified Bessel
functions have unique monotonicity properties, and that the only regular and monotonic solutions for 
the system (\ref{odew})--(\ref{W}) are those corresponding to these two particular cases, as we prove next.

\begin{theorem}
\label{doublemon}
Let $\nu\ge 1$, then the function
$$
W_{t,\nu}(x)=\Frac{(\cos\alpha I_{\nu-1}(x)-\sin\alpha K_{\nu-1}(x))(\cos\alpha I_{\nu+1}(x)-\sin\alpha K_{\nu+1}(x))}{(\cos\alpha I_{\nu}(x)+\sin\alpha K_{\nu}(x))^2}, 
$$
$\alpha=\pi t/2$, $t\in (-1,1]$, is regular and monotonic if and only if $t=0,1$.
\end{theorem}

\begin{proof}
The monotonicity for $\alpha=0,\pi/2$ has already been proven, and we have to prove that no other selection of 
$\alpha$ gives regular and monotonic solutions.

For the case $\alpha\in (-\pi/2,\pi/2)\setminus \{0\}$, for the reasons discussed in Lemma \ref{region},
we have that if $\alpha\in (-\pi/2,0)$ 
the denominator has exactly one zero at the value of $x=x_d$ for which $\tan\alpha=-I_{\nu}(x^{(d)})/K_{\nu}(x^{(d)})$ 
while if
$\alpha\in (0,\pi/2)$ the numerator has two zeros, one at $x=x_1^{(n)}$, with $\tan\alpha=I_{\nu-1}(x_1^{(n)})
/K_{\nu-1}(x_1^{(n)})$ 
and the other one $x=x_2^{(n)}$ such that $\tan\alpha=I_{\nu+1}(x_2^{(n)})/K_{\nu+1}(x_2^{(n)})$. 

We notice that
 necessarily $x_1^{(n)}\neq x_2^{(n)}$ because 
\begin{equation}
\label{wro}
K_{\nu+1}(z)I_{\nu-1}(z)-K_{\nu-1}(z)I_{\nu+1}(z)=2\nu/z^2,
\end{equation} 
as 
can be checked by considering the Wronskian relation \cite[10.28.2]{Olver:2010:BF}
\begin{equation}
\label{wro}
K_{\nu+1}(z)I_{\nu}(z)+K_{\nu}(z)I_{\nu+1}(z)=1/z,
\end{equation}
 and using the recurrence relation (\ref{TTRR})
to eliminate $K_{\nu}(z)$ and $I_{\nu}(z)$. On the other hand (\ref{wro}) also shows that 
$$I_{\nu-1}(z)/K_{\nu-1}(z)>I_{\nu+1}(z)/K_{\nu+1}(z),$$ which together with the fact that these ratios
of Bessel functions are increasing implies that $x_1^{(n)}<x_2^{(n)}$.

We note that for $x>x_2^{(n)}$, $W_{t,\nu}(x)>0$, 
because as $x\rightarrow +\infty$, $W_{t,\nu}(x)\sim W_{0,\nu}(x)>0$
for all $t$. 
Therefore $W_{t,\nu}(x)<0$ in $(x_1^{(n)},x_2^{(n)})$ and 
$W_{t,\nu}(x_1^{(n)})=W_{t,\nu}(x_2^{(n)})=0$. Then, by Rolle's theorem there exist $x_m\in (x_1^{(n)},x_2^{(n)})$ 
where $W'_{t\nu}(x_m)=0$, and in addition $W_{t,\nu}(x_m)<0$; then, necessarily $x_m$ must be such that 
$W_{t,\nu}(x_m)=w^{(O)}(x_m)<0$, and $W_{t,\nu}(x)$ reaches its absolute minimum at $x_m$, and of course it is 
not monotonic.
\end{proof}

\begin{remark}
That only two monotonic and regular solutions exists is no longer true for smaller values of $\nu$. For instance,
 it is easy to check that for $\nu\in (0,1)$ the solutions satisfying $\tan\alpha \le  \Frac{2}{\pi}\sin (\pi\nu)$
are monotonic.
\end{remark}

We end this section by obtaining some inequalities for the product of Bessel functions 
which are a direct consequence of the previous 
trigonometric bounds,
and we propose a conjecture.

\begin{theorem}
The following bound holds for $\nu\ge 0$:
\begin{equation}
I_{\nu}(x)K_{\nu}(x)>\Frac{\sqrt{3}}{2g_\nu (x)
\sin\left(\frac13\arccos\left(
\Frac{h_\nu (x)}{g_\nu (x)^3}\right)+\Frac{\pi}{3}\right)}
\end{equation}
where $g_\nu (x)=\sqrt{3(\nu^2+x^2)+1}$, $h_{\nu}(x)=9\nu^2-\Frac{9}{2}x^2-1$.
\end{theorem}

\begin{proof}
Adding the bounds of Corollary \ref{cotastrig}
$$
\Frac{I_{\nu-1}(x)}{I_{\nu}(x)}+\Frac{K_{\nu-1}(x)}{K_{\nu}(x)}<
\Frac{2g_\nu (x)}{\sqrt{3}x}\cos\left(\Frac{1}{3}\arccos\left(
\Frac{h_{\nu}(x)}{g_\nu (x)^3}
\right)-\Frac{\pi}{6}\right).
$$
This sum is positive because both summands are.
Now, using (\ref{prodre}) the result follows.
\end{proof}

As we later check, this is a bound which is very sharp in the three limits $x\rightarrow 0^+$, 
$x\rightarrow +\infty$ and $\nu\rightarrow +\infty$. Of course, a simpler but less sharp bound can
be established by bounding the sine function by $1$. Then, we have:

\begin{corollary}
For $\nu\ge 0$ the following holds
$$
I_{\nu}(x)K_{\nu}(x)>\Frac{1}{2\sqrt{x^2+\nu^2+\frac13}}
$$
\end{corollary}

The previous bound is sharp as $x\rightarrow +\infty$ (see (\ref{p1})) but not as $x\rightarrow 0^+$. The 
factor $2$ in the denominator can not be changed without losing the sharpness, and it is the best possible constant
in this sense. However, the summand $1/3$ inside the root can be lowered. We have checked numerically 
the validity of the 
following result, for which we have no proof so far:
\begin{conjecture}
For $x>0$ and $\nu\ge -1$ the following bound holds:
$$
I_{\nu}(x)K_{\nu}(x)>\Frac{1}{2\sqrt{x^2+\nu^2+\Frac15}},
$$
and this is the best possible bound of the form $a/\sqrt{x^2+\nu^2+b}$.
\end{conjecture}

\subsection{Sharpness of the bounds}

For analyzing the sharpness of these bounds we denote by $U_{\nu}^{(I)}(x)$ the upper bound for 
$I_{\nu-1}(x)/I_{\nu}(x)$ of the previous corollary, and by $U_{\nu}^{(K)}(x)$ the corresponding bound
for $K_{\nu-1}(x)/K_{\nu}(x)$ and define the relative accuracy
$$
\epsilon_{\nu}^{(I)}(x)=\Frac{U_{\nu}^{(I)}(x)I_\nu (x)}{I_{\nu-1} (x)}-1
$$
and similarly for the second kind Bessel function. Considering the expansions detailed in the Appendix we have
the following results:

\begin{corollary} The upper bound for the ratio of modified Bessel functions of the first kind, 
$U_\nu^{(I)}(x)$, 
is very sharp as $x\rightarrow +\infty$, $x\rightarrow 0^+$ and 
 $\nu \rightarrow 
+\infty$ in the sense that
$$
\epsilon_{\nu}^{(I)}(x)=\Frac{1}{4x^2}+{\cal O}(x^{-3}),\,x\rightarrow +\infty
$$
$$
\epsilon_{\nu}^{(I)}(x)=\Frac{x^4}{8\nu^2 (\nu+1)^3 (\nu+2)}+{\cal O}(x^6),\,x\rightarrow 0^+
$$
$$
\epsilon_{\nu}^{(I)}(x)=\Frac{x^4}{8\nu^6}+{\cal O}(\nu^{-7}),\,\nu\rightarrow +\infty
$$
\end{corollary}

\begin{corollary}
\label{plastc}
The relative accuracy for the upper bound for the modified Bessel function of the second kind verifies:
$$
\epsilon_{\nu}^{(K)}(x)=\Frac{1}{4x^2}+{\cal O}(x^{-3}),\,x\rightarrow +\infty
$$
$$
\epsilon_{\nu}^{(K)}(x)={\cal O}(x^p),\,p=\min\{2(\nu-1),2\},\,\nu>1,\,\nu\notin {\mathbb N},\, x\rightarrow 0^+
$$
$$
\epsilon_{\nu}^{(K)}(x)={\cal O}(x^{-2\nu}),\,0\le\nu<1,\,x\rightarrow 0^+
$$
$$
\epsilon_{\nu}^{(K)}(x)=\Frac{x^2}{2\nu^4}+{\cal O}(\nu^{-5}),\,\nu\rightarrow +\infty
$$
\end{corollary}

\begin{corollary}
\label{lastc}
The relative accuracy for the product $I_{\nu}(x)K_{\nu}(x)$ verifies:
$$
\epsilon_{\nu}^{(P)}(x)=\Frac{1}{4x^2}+{\cal O}(x^{-3}),\,x\rightarrow +\infty
$$
$$
\epsilon_{\nu}^{(P)}(x)={\cal O}(x^p),\,p=\min\{2\nu,4\},\,\nu>0,\,,\nu\notin {\mathbb N},\,x\rightarrow 0^+
$$
$$
\epsilon_{\nu}^{(K)}(x)=\Frac{x^4}{4\nu^6}+{\cal O}(\nu^{-7}),\,\nu\rightarrow +\infty
$$
\end{corollary}

\begin{remark}
The errors for the expansions as $x\rightarrow 0^+$ in Corollaries \ref{plastc} and \ref{lastc} must be
multiplied by $\log x$ when $\nu\in{\mathbb N}$. See the Appendix.
\end{remark}

\section*{Appendix}

The ratios and double ratios
\begin{equation}
\Phi_{\nu}(x)=\Frac{\cali_{\nu-1}(x)}{\cali_\nu (x)},\,W_{\nu}(x)=\Frac{\Phi_{\nu}(x)}{\Phi_{\nu+1}(x)}
\end{equation}
have, using \cite[10.4.1-2]{Olver:2010:BF}, the following expansions as $x\rightarrow +\infty$:

\begin{equation}
\label{expin}
\pm\Frac{\cali_{\nu-1}(x)}{\cali_\nu (x)}=1\pm\Frac{\nu-1/2}{x}+\Frac{\nu^2-1/4}{2x^2}+{\cal O}(x^{-3})
\end{equation}

\begin{equation}
\label{expanin}
\Frac{\cali_{\nu-1}(x)\cali_{\nu+1}(x)}{\cali_\nu (x)^2}=1\mp \Frac{1}{x}\pm\Frac{\nu^2-1/4}{2 x^3}+{\cal O}(x^{-4})
\end{equation}
where the upper sign corresponds to $\cali_\nu (x)=I_\nu (x)$ and the lower sign to 
$\cali_\nu (x)=e^{i\i\nu} K_\nu (x)$.

Using 
\cite[10.25.2]{Olver:2010:BF}, 
the Maclaurin series for the regular solution at $x=0$, $I_{\nu}(x)$, are, for $\nu\ge 0$:

\begin{equation}
\label{serin}
x\Frac{I_{\nu-1}(x)}{I_\nu (x)}=2\nu+\Frac{x^2}{2(\nu+1)}-\Frac{x^4}{8(\nu+1)^2 (\nu+2)}+{\cal O}(x^6)
\end{equation}

\begin{equation}
\label{seriesI}
\Frac{I_{\nu-1}(x)I_{\nu+1}(x)}{I_\nu (x)^2}= \Frac{\nu}{\nu+1}+\Frac{x^2}{2(\nu+1)^3}+{\cal O}(x^4)
\end{equation}

With respect to the modified Bessel function of the second kind as $x\rightarrow 0^+$, because 
\begin{equation}
\label{defk}
K_{\nu}(x)=\Frac{\pi}{2}\Frac{I_{-\nu}(x)-I_{\nu}(x)}{\sin(\nu\pi)},
\end{equation}
we have that for $\nu>1$, $\nu\notin {\mathbb N}$
\begin{equation}
\label{serk1}
x\Frac{K_{\nu-1}(x)}{K_\nu (x)}=-x\Frac{I_{1-\nu}(x)}{I_{-\nu}(x)}(1+{\cal O}(x^{2(\nu-1)}))
\end{equation}
and
\begin{equation}
\label{serk2}
-x\Frac{I_{1-\nu}(x)}{I_{-\nu}(x)}=\Frac{x^2}{2(\nu-1)}-\Frac{1}{8(\nu-1)^2(\nu-2)}x^4
+{\cal O}(x^8)
\end{equation} 
while for $0<\nu< 1$
\begin{equation}
x\Frac{K_{\nu-1}(x)}{K_\nu (x)}={\cal O}(x^{2\nu}).
\end{equation}

For integer $\nu$ a logarithmic term enters the expansions for $xK_{\nu-1}(x)/K_\nu (x)$ and
for instance we have $xK_{0}(x)/K_1 (x)={\cal O}(x\log x)$, $xK_{2}(x)/K_1 (x)=x^2/2(1+{\cal O}(x^2\log(x)))$. 
For $\nu=n\in {\mathbb N}$, 
the first terms in the expansion of $xK_{\nu-1}(x)/K_{\nu}(x)$ 
are given by the first $n-1$ terms in (\ref{serk2}), which are well defined, 
and a logarithmic factor must be added
 to the error term in (\ref{serk1}).

Using \cite[10.41.3-4]{Olver:2010:BF} we get that as $\nu\rightarrow +\infty$ for $x$ fixed:
$$
x\Frac{I_{\nu-1}(x)}{I_\nu (x)}=2\nu+\Frac{x^2}{2\nu}-\Frac{x^2}{2\nu^2}-\Frac{x^4-4x^2}{8\nu^3}
+\Frac{x^4-x^2}{2\nu^4}+{\cal O}(\nu^{-5}),
$$
$$
x\Frac{K_{\nu-1}(x)}{K_\nu (x)}=\Frac{x^2}{2\nu}+\Frac{x^2}{2\nu^2}-\Frac{x^4-4x^2}{8\nu^3}
-\Frac{x^4-x^2}{2\nu^4}+{\cal O}(\nu^{-5})
$$

As for the product $I_{\nu}(x) K_{\nu}(x)$ we have
\begin{equation}
\label{p1}
I_{\nu}(x)K_{\nu}(x)=\Frac{1}{2x}-\Frac{\nu^2-1/4}{4x^3}+{\cal O}(x^{-5}),\,x\rightarrow +\infty,
\end{equation}
\begin{equation}
\label{p2}
\begin{array}{ll}
I_{\nu}(x)K_{\nu}(x)&=\left(\Frac{\pi}{2\sin(\pi\nu)}I_{\nu}(x)I_{-\nu}(x)\right)(1+{\cal O}(x^{2\nu}))\\
&\hspace*{-0.5cm}=\left(\Frac{1}{2\nu}-\Frac{1}{4}\Frac{x^2}{\nu(\nu^2-1)}+\ldots\right)(1+{\cal O}(x^{2\nu})),\,x\rightarrow 0^+,\,\nu\notin {\mathbb N}
\end{array}
\end{equation}
\begin{equation}
\label{p3}
I_{\nu}(x)K_{\nu}(x)=\Frac{1}{2\nu}-\Frac{x^2}{4\nu^3}+{\cal O}(\nu^{-5}),
\end{equation}
where in the limit $x\rightarrow 0^+$ and $\nu\in {\mathbb N}$ similar modifications as that considered
 after (\ref{serk2}) should be taken into account.

\section*{Acknowledgements}
The author acknowledges support from Ministerio de Ciencia e Innovaci\'on, project PGC2018-098279-B-I00
(MCIU/AEI/FEDER, UE).


\bibliography{trigono}

\end{document}